\documentclass[10pt]{article}

\usepackage{ucs}
\usepackage{amssymb}
\usepackage{amsthm}
\usepackage{amsmath}
\usepackage{latexsym}
\usepackage[cp1251]{inputenc}
\usepackage[english]{babel}
\usepackage{graphicx}
\usepackage{wrapfig}
\usepackage{caption}
\usepackage{subcaption}
\usepackage{txfonts}
\usepackage{lmodern}
\usepackage{mathrsfs}
\usepackage{eufrak}
\usepackage{algorithmicx}
\usepackage{algpseudocode}
\usepackage{algorithm}
\usepackage{dsfont}
\usepackage{enumerate}
\usepackage{hyphenat}
\usepackage{hyperref}

\DeclareMathOperator{\vbl}{\mathsf{vbl}}
\DeclareMathOperator{\dom}{\mathsf{dom}}

\newcommand{\func}{\zeta}

\usepackage[left=3cm,right=3cm,top=1.9cm,bottom=2.4cm,bindingoffset=0cm]{geometry}

\title{The Local Action Lemma}
\date{}
\author{Anton Bernshteyn\thanks{Supported by the Illinois Distinguished Fellowship.}\\ University of Illinois at Urbana-Champaign}

\newtheorem{theo}{Theorem}[section]

\newtheorem{lemma}[theo]{Lemma}

\newtheorem{corl}[theo]{Corollary}
\newtheorem{conj}[theo]{Conjecture}
\newtheorem{claim}[theo]{Claim}

\theoremstyle{definition}
\newtheorem{defn}[theo]{Definition}

\theoremstyle{remark}
\newtheorem{remk}[theo]{Remark}

\def\keywords{\vspace{.2em}
{\textit{Keywords}:\,\relax%
}}
\def\endkeywords{\par}

\newcommand{\beauty}[1]{\mathsf{#1}}

\begin{document}
	\maketitle
	
	\begin{abstract}
		The Lov\'{a}sz Local Lemma is a very powerful tool in probabilistic combinatorics, that is often used to prove existence of combinatorial objects satisfying certain constraints. Moser and Tardos \cite{Moser} have shown that the LLL gives more than just pure existence results: there is an effective randomized algorithm that can be used to find a desired object. In order to analyze this algorithm Moser and Tardos developed the so-called \emph{entropy compression method}. It turned out that one could obtain better combinatorial results by a direct application of the entropy compression method rather than simply appealing to the LLL. We provide a general statement that implies both these new results and the LLL itself.
	\end{abstract}
	
	
	\tableofcontents
	
	\section{Introduction}
	
	Suppose that we are given a family $\beauty{A}$ of \emph{bad} random events in some probability space. How can we show that with positive probability none of the events $\beauty{A}$ happen? If the events $\beauty{A}$ are mutually independent, then $\Pr\left(\bigcap_{A\in\beauty{A}}\overline{A}\right)=\prod_{A\in\beauty{A}}\left(1-\Pr(A)\right)$, which is positive if the individual probabilities of all events $A\in\beauty{A}$ are less than $1$. Although the events under consideration are not usually independent, their dependence is often limited or structured somehow. The famous Lov\'{a}sz Local Lemma (the LLL for short) provides a useful way to capture this limitation.
	
	\begin{theo}[Lov\'{a}sz Local Lemma \cite{AS}]
		Let $\beauty{A}$ be a finite set of random events in a probability space $\Omega$. For $A\in\beauty{A}$ let $\Gamma(A)$ be a subset of $\beauty{A}\setminus\{A\}$ such that $A$ is independent from the $\sigma$-algebra generated by $\beauty{A}\setminus(\Gamma(A)\cup\{A\})$. Suppose that there exists an assignment of reals $\mu:\beauty{A}\to[0;1)$ such that for every $A\in\beauty{A}$ we have
		\begin{equation}\label{eq:LLL}
			\Pr(A)\leq \mu(A) \prod_{B\in \Gamma(A)}(1-\mu(B)).
		\end{equation}
		Then $\Pr\left(\bigcap_{A\in\beauty{A}}\overline{A}\right) > 0$.
	\end{theo}	
	
	Note that the condition (\ref{eq:LLL}) employs only ``local'' information about an event $A$, i.e. it refers only to the events in $\Gamma(A)$ (hence the name ``Local Lemma''). Therefore, the LLL can be applied even if the whole dependency structure of $\beauty{A}$ is unknown or very hard to describe (which is almost always the case). This property allowed the LLL to become one of the most important tools in probabilistic combinatorics.
	
	In their paper \cite{Bissacot} Bissacot \emph{et al.} strengthened the LLL by replacing the condition (\ref{eq:LLL}) with a weaker (and more complicated) one. This Improved LLL was successfully applied in \cite{Ndreca} by Ndreca \emph{et al.} to improve many combinatorial results previously obtained using the LLL. Another improvement to the LLL was made by Kolipaka \emph{et al.} \cite{Kolipaka}, who established a series of LLL-like statements that require more and more information about the dependency structure of $\beauty{A}$. One of these statements, the so-called Clique LLL, was applied in \cite{Kolipaka} to improve bounds for the acyclic edge coloring problem and the non-repetitive coloring problem.
	
	In most applications of the LLL the events under consideration are determined by a set of mutually independent discrete random variables $\beauty{V}$. If each event $A \in \beauty{A}$ is determined by a subset $\vbl(A) \subseteq \beauty{V}$, then we can let $\Gamma(A)$ to be $\{B \in \beauty{A}\setminus\{A\}\,:\,\vbl(B)\cap\vbl(A) \neq \emptyset\}$. This setting for the LLL is called the \emph{variable framework}. A major breakthrough was made by Moser and Tardos \cite{Moser}, who showed that in the variable framework there exists a simple Las Vegas algorithm with expected polynomial runtime that searches the probability space for a point which avoids all the bad events. It was later shown by Pegden \cite{Pegden} and Kolipaka and Szegedy \cite{Kolipaka1} that this algorithm is efficient even under the conditions weaker than those of the LLL.
	
	In order to analyze this algorithm Moser and Tardos developed the so-called \emph{entropy compression method}. Roughly speaking, the method consists in two main stages. First, we provide some way to encode an execution process of the algorithm so that the outcomes of all random choices performed by the algorithm can be uniquely recovered from the resulting encoding.  On the second stage we use the structure of this encoding to show that if the expected runtime of the algorithm were unbounded then this encoding would losslessly compress the original random data while reducing its Shannon entropy, which is impossible. 
		
	It was discovered lately (and somewhat unexpectedly) that one can obtain better combinatorial results by a \emph{direct application} of the entropy compression method rather than simply appealing to the LLL. The idea was to construct a randomized procedure that solved a particular combinatorial problem (instead of proving the LLL in general) and then apply an entropy compression argument to show that this procedure had expected finite runtime. Examples can be found in \cite{Duj}, \cite{Esperet}, \cite{Goncalves} etc., and some of them are discussed in more details in Section \ref{sec:applications}.
	
	Note that the entropy compression method is indeed a ``method'' that one can use to attack a problem rather than a general theorem that contains various combinatorial results as its special cases. Our goal here is to fill this gap and provide a generalization of the LLL, that, in particular, implies the new combinatorial results obtained using the entropy compression method. We formulate the new lemma in Section \ref{sec:LAL} and give a proof for it in Section \ref{sec:proof}. Section \ref{sec:applications} is devoted to its applications, highlighting different aspects of using the lemma. There we show how entropy compression proofs can be ``translated'' into the language of probability and then derived from the lemma, and why the LLL is a particular case of it. Some of the examples given in Section \ref{sec:applications} have less combinatorial flavor. They are intended to show that our lemma can be applied in situations far from the usual LLL-type results. 
	
	\section{The Local Action Lemma}\label{sec:LAL}
	
	In this section we present the main result of this paper. We call it the \emph{Local Action Lemma} (the LAL for short) for the reasons that will become apparent after the statement of the theorem.
	
	Recall that a \emph{monoid} is an algebraic structure with a single associative binary operation and an identity element. We will use the multiplicative notation for the monoid operation, and denote its identity element by $1$. For a monoid $M$, a subset $B \subseteq M$ is called a \emph{generating set} of $M$ if $M$ is the smallest set containing $B$ that is closed under the monoid operation. If $B$ is a fixed generating set of $M$, then we refer to the elements of $B$ as the \emph{generators} of $M$.
	
	A (left) \emph{action} of a monoid $M$ on a set $\mathcal{X}$ is a map $\varphi : M\times \mathcal{X} \to \mathcal{X}$ that is compatible with the monoid operation, i.e. 
	$$\varphi(\alpha\beta, x) = \varphi(\alpha, \varphi(\beta, x)),$$
	and
	$$\varphi(1, x) = x$$
	for all $\alpha$, $\beta \in M$ and $x \in \mathcal{X}$. If a monoid action $\varphi$ is fixed, then we say that $M$ \emph{acts} on $\mathcal{X}$ and use the notation $\alpha.x$ for $\varphi(\alpha, x)$.
	
	Let $M$ be a monoid with a generating set $B$. If $\func : B \to \mathbb{R}_+$, then for each element $\alpha \in M$ let
	\begin{equation}\label{eq:underline}
		\underline{\func}(\alpha) \coloneqq \inf \left\{\prod_{i=1}^k\func(\beta_i) \, : \, \beta_1, \ldots, \beta_k \in B,\; \prod_{i=1}^k \beta_i = \alpha\right\}.
	\end{equation}
	In most applications it is safe to assume that the infimum in (\ref{eq:underline}) is attained at some particular $\beta_1$, \ldots, $\beta_k \in B$. Nevertheless, that may not be the case if $B$ is infinite, which will lead to some technical issues in the proof. 
	
	Suppose that $\mathcal{X}$ is a non-empty set equipped with a $\sigma$-algebra $\Sigma$. If a monoid $M$ acts on $\mathcal{X}$, then this action is \emph{measurable} if for every $\alpha \in M$ the map $x \mapsto \alpha.x$ is measurable. We also use the following notational convention for conditional probabilities: If $F$ is a random event and $\Pr(F) = 0$, then $\Pr(E\vert F) = 0$ for all events $E$. Now we are ready to state the LAL.
	
	\begin{theo}[Local Action Lemma]\label{theo:LAL}
		Suppose that $\mathcal{X}$ is a non-empty set equipped with a $\sigma$-algebra $\Sigma$, and let $\mathcal{G} \in \Sigma$. Let $M$ be an at most countable monoid with a generating set $B \subseteq M$. Suppose that $M$ acts measurably on $\mathcal{X}$, and for every $x \in \mathcal{G}$ and $\alpha \in M$ we have $\alpha.x \in \mathcal{G}$.
		
		Let $\Omega$ be a probability space, and let $X : \Omega \to \mathcal{X}$ be a random variable. For every $x \in \mathcal{X}\setminus\mathcal{G}$ and $\beta \in B$ such that $\beta.x \in \mathcal{G}$ choose an arbitrary element $g_\beta(x) \in M$ in such a way that the maps $x \mapsto g_\beta(x)$ are measurable. For $\beta \in B$ and $\alpha \in M$ define
		$$P(\beta, \alpha) \coloneqq \sup_{\gamma \in M}
			\Pr\left(g_\beta(\gamma . X) = \alpha \,\middle\vert\, \alpha\beta\gamma . X \in \mathcal{G}\right).$$
	
		Suppose that for some $\alpha \in M$ we have $\Pr(\alpha. X \in \mathcal{G}) > 0$. If there exists a function $\func : B \to \mathbb{R}_+$ such that for every $\beta \in B$ we have
		\begin{equation}\label{eq:LAL}
			\func(\beta)\geq 1 + \sum_{\alpha \in M} P(\beta, \alpha) \underline{\func}(\alpha\beta),
		\end{equation}
		then $\Pr(X \in \mathcal{G}) > 0$.
	\end{theo}
	
	\section{Proof of the LAL}\label{sec:proof}
	
	Let $B^*$ be the free monoid on the set $B$. Then there exists a unique homomorphism $\tau : B^* \to M$ such that $\tau(\beta) = \beta$ for all $\beta \in B$. If $w = \beta_1 \ldots \beta_k \in B^*$, then let
	$$\func(w) \coloneqq \prod_{i=1}^k\func(\beta_i).$$
	Hence for every $\alpha \in M$ we have
	$$\underline{\func}(\alpha) = \inf_{w \in \tau^{-1}(\alpha)} \func(w).$$
	
	Note that the action of $M$ on $\mathcal{X}$ induces an action of $B^*$ on $\mathcal{X}$, namely if $w \in B^*$ and $x \in \mathcal{X}$, then~$w.x \coloneqq \tau(w).x$.
	
	Suppose that $\Pr(X \in \mathcal{G}) = 0$. For every $\alpha \in M$ and for every $\varepsilon > 0$ fix an element $\tau^{-1}_\varepsilon (\alpha) \in \tau^{-1}(\alpha)$ such that $\func(\tau^{-1}_\varepsilon(\alpha)) < (1+\varepsilon)\underline{\func}(\alpha)$. Consider the following randomized process, that generates an infinite random sequence $(w_0, w_1, w_2, \ldots)$ of elements of $B^*$ such that $\Pr(w_k.X \in \mathcal{G}) > 0$ for every index~$k$, and an infinite random sequence $(\alpha_1, \alpha_2, \ldots)$ of elements of $M \cup \{*\}$ (where $*\not\in M$).
	
	Let $w_0 \in B^*$ be such that $\Pr(w_0 . X \in \mathcal{G}) > 0$ (the existence of such $w_0$ is guaranteed by the conditions of the LAL). Now suppose that $(w_0, \ldots, w_{k-1})$ and $(\alpha_1, \ldots, \alpha_{k-1})$ are already defined. Note that $w_{k-1} \neq 1$, because $\Pr(X \in \mathcal{G}) = 0$. So let $w_{k-1} = \beta v$ for some $\beta \in B$ and $v \in B^*$. Choose a random element $x \in \{y \in \mathcal{X}\, :\, \beta.y \in \mathcal{G}\}$ according to the distribution of $v.X$ conditioned on the event $\{\beta v.X \in \mathcal{G}\}$. If $x \in \mathcal{G}$, then $\Pr(v.X \in \mathcal{G}) > 0$, so we can let $w_k \coloneqq v$ and $\alpha_k \coloneqq *$. Otherwise take $\alpha_k \coloneqq g_\beta (x)$. Choose some positive value for $\varepsilon$ (we will discuss the way of choosing it later). Observe that
	$$\Pr(\tau^{-1}_\varepsilon(\alpha_k \beta) v.X \in\mathcal{G}) = \Pr(\alpha_k\beta\tau(v).X \in \mathcal{G}) \geq \Pr(\beta v.X \in \mathcal{G}) > 0,$$
	so we can let $w_k \coloneqq \tau^{-1}_\varepsilon(\alpha_k \beta) v$. 
	
	\begin{defn}
		Let $(w_0, \ldots, w_n)$ be a sequence of elements of $B^*$ such that $\Pr(w_k.X \in \mathcal{G}) > 0$ for every $0\leq k\leq n$, and let $(\alpha_1, \ldots, \alpha_n)$ be a sequence of elements of $M \cup \{*\}$. For every $1 \leq k \leq n$ let $w_{k-1} = \beta_k v_k$, where $\beta_k \in B$ and $v_k \in B^*$. Suppose that for every $1 \leq k \leq n$ either $\alpha_k = *$ and $w_k = v_k$, or $\alpha_k \in M$ and $w_k = u_k v_k$ for some $u_k \in \tau^{-1}(\alpha_k \beta_k)$. Then the pair of sequences $(w_0, \ldots, w_n)$, $(\alpha_1, \ldots, \alpha_n)$ is called \emph{admissible}. We will usually write $(w_0, \ldots, w_n; \alpha_1, \ldots, \alpha_n)$ to denote an admissible pair.
	\end{defn}
	
	It is clear that the first entries of the sequences produced by the algorithm described above always form an admissible pair. Lemma \ref{lemma:prbound} gives an upper bound on the probability for a given admissible pair to be produced in this way.
	
	\begin{lemma}\label{lemma:prbound}
		Let $(w_0, \ldots, w_n; \alpha_1, \ldots, \alpha_n)$ be an admissible pair. Then the probability that the first entries of the sequences produced by the algorithm described above are $(w_0, \ldots, w_n)$ and $(\alpha_1, \ldots, \alpha_n)$ is at most
			\begin{equation}\label{eq:prbound}
				\frac{\Pr(w_n.X \in \mathcal{G})}{\Pr(w_0.X\in\mathcal{G})} \prod_{\substack{1\leq k \leq n:\\ \alpha_k \in M}} P(\beta_k, \alpha_k).
			\end{equation}
			In particular, this probability is at most
			\begin{equation*}
				\frac{1}{\Pr(w_0.X\in\mathcal{G})} \prod_{\substack{1\leq k \leq n:\\ \alpha_k \in M}} P(\beta_k, \alpha_k).
			\end{equation*}			
	\end{lemma}
	\begin{proof}
		The proof is by induction on $n$. If $n = 0$, then (\ref{eq:prbound}) is just equal to $1$, so the statement is obvious in this case. Now suppose that the probability that the first entries of the sequences produced by the algorithm are $(w_0, \ldots, w_{n-1})$ and $(\alpha_1, \ldots, \alpha_{n-1})$ is at most
		$$\frac{\Pr(w_{n-1}.X \in \mathcal{G})}{\Pr(w_0.X\in\mathcal{G})} \prod_{\substack{1\leq k \leq n-1:\\ \alpha_k \in M}} P(\beta_k, \alpha_k).$$
		
		Denote by $x$ the element that was randomly chosen on the $n$'th step of the algorithm from $\{y \in \mathcal{X}\, :\, \beta_n.y \in \mathcal{G}\}$ according to the distribution of $v_n.X$ conditioned on the event $\{\beta_n v_n.X \in \mathcal{G}\}$. Consider the two cases. 
		
		{\sc Case 1:} $\alpha_n = *$. It means that $x \in \mathcal{G}$. The probability of this is
		$$\Pr(v_n.X \in \mathcal{G}\,\vert\,\beta_n v_n.X \in \mathcal{G}) = \frac{\Pr(v_n.X \in \mathcal{G})}{\Pr(\beta_n v_n.X \in \mathcal{G})} = \frac{\Pr(w_n.X \in \mathcal{G})}{\Pr(w_{n-1}.X \in \mathcal{G})}.$$
		Therefore, the probability that the first entries of the sequences produced by the algorithm are $(w_0, \ldots, w_n)$ and $(\alpha_1, \ldots, \alpha_n)$ is at most
		\begin{align*}
			\frac{\Pr(w_n.X \in \mathcal{G})}{\Pr(w_{n-1}.X \in \mathcal{G})} \cdot \frac{\Pr(w_{n-1}.X \in \mathcal{G})}{\Pr(w_0.X\in\mathcal{G})} \prod_{\substack{1\leq k \leq n-1:\\ \alpha_k \in M}} P(\beta_k, \alpha_k) = \frac{\Pr(w_n.X \in \mathcal{G})}{\Pr(w_0.X\in\mathcal{G})} \prod_{\substack{1\leq k \leq n:\\ \alpha_k \in M}} P(\beta_k, \alpha_k).
		\end{align*}
		
		{\sc Case 2:} $\alpha_n \in M$. Then $x \not \in \mathcal{G}$ and $g_{\beta_n}(x) = \alpha_n$. The probability of this is
		$$\Pr(g_{\beta_n}(v_n.X) = \alpha_n\,\vert\, \beta_n v_n.X \in \mathcal{G}) = \frac{\Pr(g_{\beta_n}(v_n.X) = \alpha_n)}{\Pr(\beta_n v_n.X \in \mathcal{G})} = \frac{\Pr(g_{\beta_n}(v_n.X) = \alpha_n)}{\Pr(w_{n-1}.X \in \mathcal{G})}.$$
		Note that $w_n.X = \tau(w_n).X = \alpha_n \beta_n \tau(v_n).X$. In particular, if $\beta_n v_n.X \in \mathcal{G}$, then $w_n.X \in \mathcal{G}$. Thus
		\begin{align*}
			\Pr(g_{\beta_n}(v_n.X) = \alpha_n) &= \Pr(g_{\beta_n}(v_n.X) = \alpha_n \, \vert \, w_n.X \in \mathcal{G})\cdot \Pr(w_n.X \in \mathcal{G})\\
			&= \Pr(g_{\beta_n}(\tau(v_n).X) = \alpha_n \, \vert \, \alpha_n \beta_n \tau(v_n).X \in \mathcal{G})\cdot \Pr(w_n.X \in \mathcal{G})\\
			&\leq P(\beta_n, \alpha_n)\cdot \Pr(w_n.X \in \mathcal{G}). 
		\end{align*}
		Therefore, the probability that the first entries of the sequences produced by the algorithm are $(w_0, \ldots, w_n)$ and $(\alpha_1, \ldots, \alpha_n)$ is at most
		\begin{align*}
			\frac{P(\beta_n, \alpha_n)\cdot \Pr(w_n.X \in \mathcal{G})}{\Pr(w_{n-1}.X \in \mathcal{G})} \cdot \frac{\Pr(w_{n-1}.X \in \mathcal{G})}{\Pr(w_0.X\in\mathcal{G})} \prod_{\substack{1\leq k \leq n-1:\\ \alpha_k \in M}} P(\beta_k, \alpha_k) = \frac{\Pr(w_n.X\in\mathcal{G})}{\Pr(w_0.X\in\mathcal{G})} \prod_{\substack{1\leq k \leq n:\\ \alpha_k \in M}} P(\beta_k, \alpha_k).
		\end{align*}
	\end{proof}
	
	For an admissible pair $(w_0, \ldots, w_n; \alpha_1, \ldots, \alpha_n)$ denote
	$$P(w_0, \ldots, w_n; \alpha_1, \ldots, \alpha_n) \coloneqq \prod_{\substack{1\leq k \leq n:\\ \alpha_k \in M}} P(\beta_k, \alpha_k).$$
	
	\begin{defn}
		An \emph{extended admissible pair} is a sequence $(w_0, \ldots, w_{n-1}; \alpha_1, \ldots, \alpha_n)$, where $(w_0$, \ldots, $w_{n-1};$ $\alpha_1, \ldots, \alpha_{n-1})$ is an admissible pair and $\alpha_n \in M$. Denote the set of all extended admissible pairs by $E$.
	\end{defn}
	
	Now we have to discuss how to choose $\varepsilon$ on each step of the algorithm. Note that before defining $\varepsilon$ on the $k$'th step of the algorithm we have already constructed $(w_0, \ldots, w_{k-1})$ and $(\alpha_1, \ldots, \alpha_{k})$. Moreover, $(w_0, \ldots, w_{k-1}; \alpha_0, \ldots, \alpha_k)$ form an extended admissible pair. So we can think of $\varepsilon$ as a function $\varepsilon : E \to \mathbb{R}_+$.
	
	\begin{defn}
		Let $\varepsilon : E \to \mathbb{R}_+$. An \emph{$\varepsilon$-admissible} pair is an admissible pair $(w_0, \ldots, w_n; \alpha_1, \ldots, \alpha_n)$ such that $w_k = \tau^{-1}_{\varepsilon_k} (\alpha_k \beta_k) v_k$, where $\varepsilon_k = \varepsilon(w_0, \ldots, w_{k-1}; \alpha_1, \ldots, \alpha_k)$, for every index $k$ for which $\alpha_k \in M$. 
	\end{defn}
	
		Note that if the algorithm uses a function $\varepsilon : E \to \mathbb{R}_+$, then it can produce only $\varepsilon$-admissible sequences.
	
		The following definition employs the observation that the elements of $B^*$ are simply the finite sequences of elements of $B$. In particular, we can define the length $|w|$ of any element $w \in B^*$.
		
	\begin{defn}
		For $N \geq 0$, $w \in B^*\setminus\{1\}$ and $\varepsilon : E \to \mathbb{R}_+$ let $\mathcal{A}^\varepsilon_N(w)$ be the set of all $\varepsilon$-admissible pairs $(w_0, \ldots, w_n;\alpha_1, \ldots, \alpha_n)$ for which the following conditions hold:
		\begin{enumerate}
			\item $w_0 = w$;
			\item the number of indices $k$ such that $\alpha_k \in M$ is at most $N$;
			\item $|w_k| \geq |w_0|$ for all $1 \leq k \leq n$;
			\item either $n= 0$, or $\alpha_n \in M$.
		\end{enumerate}
	\end{defn}
	
	\begin{lemma}\label{lemma:meanbound}
		Suppose that $N \geq 0$ and $w \in B^*\setminus\{1\}$. Let $w = \beta v$, where $\beta \in B$ and $v \in B^*$. Then
		\begin{equation}\label{eq:meanbound}
			\inf_{\varepsilon : E \to \mathbb{R}_+} \sum_{\mathcal{A}^\varepsilon_N(w)} P(w_0, \ldots, w_n; \alpha_1, \ldots, \alpha_n)\, \leq\, \func(\beta).
		\end{equation}
	\end{lemma}
	\begin{proof}
		Let us consider another set of $\varepsilon$-admissible pairs. Namely, for $N \geq 0$, $w \in B^*\setminus\{1\}$ and $\varepsilon : E \to \mathbb{R}_+$ let $\mathcal{B}^\varepsilon_N(w)$ be the set of all $\varepsilon$-admissible pairs $(w_0, \ldots, w_n;\alpha_1, \ldots, \alpha_n)$ for which the following conditions hold:
	\begin{enumerate}
		\item $w_0 = w$;
		\item the number of indices $k$ such that $\alpha_k \in M$ is at most $N$;
		\item $|w_k| \geq |w_0|$ for all $1 \leq k \leq n$;
		\item $|w_n| = |w_0|$.
	\end{enumerate}
	
	There exists a bijection $f : \mathcal{B}^\varepsilon_N(w) \to \mathcal{A}^\varepsilon_N(w)$. Indeed, if $(w_0, \ldots, w_n; \alpha_1, \ldots, \alpha_n) \in \mathcal{B}^\varepsilon_N(w)$, then let $k$ be the greatest index for which $\alpha_k \in M$ (if $n=0$, then $k \coloneqq 0$), and let
	\begin{equation}\label{eq:bij}
		f(w_0, \ldots, w_n; \alpha_1, \ldots, \alpha_n) \coloneqq (w_0, \ldots, w_k; \alpha_1, \ldots, \alpha_k).
	\end{equation}
	It is easy to check that the map $f$ defined by (\ref{eq:bij}) is indeed a bijection between $\mathcal{B}^\varepsilon_N(w)$ and $\mathcal{A}^\varepsilon_N(w)$. Moreover, for any $(w_0, \ldots, w_n; \alpha_1, \ldots, \alpha_n) \in \mathcal{B}^\varepsilon_N(w)$ we have
	$$P(w_0, \ldots, w_n; \alpha_1, \ldots, \alpha_n) = P(f(w_0, \ldots, w_n; \alpha_1, \ldots, \alpha_n)).$$
	Therefore, it is enough to prove that
		\begin{equation}\label{eq:meanbound1}
			\inf_{\varepsilon : E \to \mathbb{R}_+} \sum_{\mathcal{B}^\varepsilon_N(w)} P(w_0, \ldots, w_n; \alpha_1, \ldots, \alpha_n)\, \leq\, \func(\beta).
		\end{equation}
	
		The proof is by induction on $N$. If $N = 0$, then for every $\varepsilon$ the set $\mathcal{B}_0^\varepsilon(w_0)$ contains only one pair, namely $(w_0; \emptyset)$. Hence the left-hand side of (\ref{eq:meanbound1}) in this case is equal to $1$, while $\func(\beta) \geq 1$ by (\ref{eq:LAL}).
		
		Now suppose that $N \geq 1$. For $w \in B^*\setminus\{1\}$ let $\mathcal{B}_N(w)$ be the set of all admissible pairs $(w_0, \ldots, w_n;\alpha_1, \ldots, \alpha_n)$ for which the following conditions hold:
	\begin{enumerate}
		\item $w_0 = w$;
		\item the number of indices $k$ such that $\alpha_k \in M$ is at most $N$;
		\item $|w_k| \geq |w_0|$ for all $1 \leq k \leq n$;
		\item $|w_n| = |w_0|$.
	\end{enumerate}
	In particular, $\mathcal{B}^\varepsilon_N(w)$ is the subset of $\mathcal{B}_N(w)$ consisting of all $\varepsilon$-admissible pairs.
	
	Consider any pair $(w_0, \ldots, w_n;\alpha_1, \ldots, \alpha_n) \in \mathcal{B}_N(w)$ with $n \geq 1$. Let $r \coloneqq |w_1| - |w_0|$. For $0 \leq s \leq r$ let $k_s$ be the least positive index such that $|w_{k_s}| = |w_1| - s$, and let $k_{r+1}\coloneqq n+1$. Clearly,
	$$1 = k_0 < k_1 < \ldots < k_r < k_{r + 1}=n+1.$$
	Suppose that $w_0 = \beta v$ and $w_1 = \beta_0 \beta_1 \cdots \beta_r v$, where $\beta$, $\beta_0$, \ldots, $\beta_r \in B$ and $v \in B^*$. Then for every $0 \leq s \leq r$
	\begin{equation*}
		w_{k_s} = \beta_s \cdots \beta_r v.
	\end{equation*}
	Moreover, for every $0 \leq s \leq r$ the pair $(w_{k_s}, \ldots, w_{k_{s+1} - 1}; \alpha_{k_s + 1}, \ldots, \alpha_{k_{s+1}-1})$ is admissible. Furthermore, the number of indices $k$ such that $\alpha_k \in M$ for this pair is strictly less than $N$, and $|w_{k_{s+1} - 1}| = |w_{k_s}|$. Thus
	\begin{equation}\label{eq:letters}
		(w_{k_s}, \ldots, w_{k_{s+1} - 1}; \alpha_{k_s + 1}, \ldots, \alpha_{k_{s+1}-1}) \in \mathcal{B}_{N-1}(w_{k_s}) = \mathcal{B}_{N-1}(\beta_s \cdots \beta_r v).
	\end{equation}
	
	By the induction hypothesis, for every $\beta v = w \in B^*$ and every $\delta > 0$ there exists $\varepsilon [ w, \delta] : E \to \mathbb{R}_+$ such that
	\begin{equation}\label{eq:ind}
		\sum_{\mathcal{B}^{\varepsilon[ w, \delta]}_{N-1}(w)} P(w_0, \ldots, w_n; \alpha_1, \ldots, \alpha_n)\, \leq\, (1+\delta)\func(\beta).
	\end{equation}
	
	Pick any $\delta > 0$. For an integer $m \geq 1$ choose $\delta_m > 0$ in such a way that $(1+\delta_m)^m \leq 1+\delta$. Define $\varepsilon : E \to \mathbb{R}_+$ in the following way.
	\begin{enumerate}
		\item For any $(w; \alpha) \in E$ let $\varepsilon(w; \alpha) \coloneqq \delta$.
		\item Suppose that $(w_0, \ldots, w_n; \alpha_1, \ldots, \alpha_{n+1}) \in E$, $n>0$, and $|w_k| \geq |w_0|$ for all $1 \leq k \leq n$. Let $r \coloneqq |w_1| - |w_0|$. For $0 \leq s \leq r$ let $k_s$ be the least positive index such that $|w_{k_s}| = |w_1| - s$. Note that the indices $k_s$ may be defined only for some first values of $s$ (because there may be no such positive index $k$ that $|w_k| = |w_1| -s$). Let $S$ be the greatest number among $0$, \ldots, $r$ for which $k_S$ is defined. Then let
		$$
			\varepsilon(w_0, \ldots, w_n; \alpha_1, \ldots, \alpha_{n+1}) \coloneqq \varepsilon[ w_{k_S}, \delta_{r+1}](w_{k_S}, \ldots, w_n; \alpha_{k_S + 1}, \ldots, \alpha_n).
		$$  
	\end{enumerate}
	For our purposes it is irrelevant how we define $\varepsilon$ on all other elements of $E$.
	
	Now, finally, consider a pair $(w_0, \ldots, w_n;\alpha_1, \ldots, \alpha_n) \in \mathcal{B}^\varepsilon_N(w)$ with $n > 0$. Recall that $r = |w_1| - |w_0| = |\tau^{-1}_\delta(\alpha_1 \beta)|-1$, $k_s$ is the least positive index such that $|w_{k_s}| = |w_1| - s$ for any $0\leq s\leq r$, and $k_{r+1} = n+1$. Also recall that $w_0 = w = \beta v$ and $w_1 = \tau^{-1}_\delta(\alpha_1 \beta) v = \beta_0 \beta_1 \cdots \beta_r v$, where $\beta$, $\beta_0$, \ldots, $\beta_r \in B$ and $v \in B^*$. Then
		\begin{equation}\label{eq:decomp}
			P(w_0, \ldots, w_n;\alpha_1, \ldots, \alpha_n) = P(\beta, \alpha_1) \prod_{s=0}^r P(w_{k_s}, \ldots, w_{k_{s+1} - 1}; \alpha_{k_s + 1}, \ldots, \alpha_{k_{s+1}-1}).
		\end{equation}
	Note that by (\ref{eq:letters}) and the construction of $\varepsilon$ we have
		\begin{equation}\label{eq:belongs}
			(w_{k_s}, \ldots, w_{k_{s+1} - 1}; \alpha_{k_s + 1}, \ldots, \alpha_{k_{s+1}-1}) \in \mathcal{B}_{N-1}^{\varepsilon[\beta_s \cdots \beta_r v,\, \delta_{r+1}]}(\beta_s \cdots \beta_r v).
		\end{equation}
		Observe that the set on the right-hand side of (\ref{eq:belongs}) is completely determined by $w$ and $\alpha_1$.
		Therefore, combining (\ref{eq:decomp}) and (\ref{eq:ind}), and using the notation $r(\alpha) \coloneqq |\tau^{-1}_\delta(\alpha \beta)|-1$ and $\beta^\alpha_0 \beta^\alpha_1 \cdots \beta^\alpha_{r(\alpha)} = \tau^{-1}_\delta(\alpha \beta)$, we~get
		\begin{align*}
			\sum_{\mathcal{B}^\varepsilon_N(w)} P(w_0, \ldots, w_n; \alpha_1, \ldots, \alpha_n) 
				\,\leq&\, 1 + \sum_{\alpha \in M} P(\beta, \alpha) \prod_{s = 0}^{r(\alpha)} (1+\delta_{r(\alpha)+1})\func(\beta_s^\alpha) \\
				=&\, 1 + \sum_{\alpha \in M} P(\beta, \alpha) (1+\delta_{r(\alpha)+1})^{r(\alpha)+1}\prod_{s = 0}^{r(\alpha)} \func(\beta_s^\alpha) \\
				\leq\, 1 + (1+\delta)^2 \sum_{\alpha \in M} P(\beta, \alpha) \underline{\func}(\alpha \beta)\, \xrightarrow[\delta \to 0]{}&\,
				1 + \sum_{\alpha \in M} P(\beta, \alpha) \underline{\func}(\alpha \beta) \,\leq\, \func(\beta),
		\end{align*}
		where the last inequality follows by (\ref{eq:LAL}).
	\end{proof}
	
	Now Theorem \ref{theo:LAL} follows almost immediately. Indeed, let $w_0$ be an element of $B^*$ such that $\Pr(w_0.X \in \mathcal{G}) > 0$ with the minimal possible length. Note that $w_0 \neq 1$, so let $w_0 = \beta v$, where $\beta \in B$ and $v \in B^*$. Choose some $\varepsilon : E \to \mathbb{R}_+$ and start the algorithm with these $w_0$ and $\varepsilon$. Suppose that it produced sequences $(w_0, w_1, \ldots)$ and $(\alpha_1, \alpha_2, \ldots)$. Then by the choice of $w_0$ we have $|w_k| \geq |w_0|$ for all $k$. In particular, if $(\alpha_{k_1}, \alpha_{k_2}, \ldots)$ is the subsequence of $(\alpha_1, \alpha_2, \ldots)$ that consists of all entries $\alpha_k \in M$, then
	$$(w_0 , \ldots, w_{k_N}; \alpha_1, \ldots, \alpha_{k_N}) \in \mathcal{A}^\varepsilon_N(w_0).$$
	Note that the sequence $(\alpha_1, \alpha_2, \ldots)$ must contain infinitely many elements from $M$, because if $\alpha_k = *$, then $|w_k| = |w_{k-1}| - 1$. Hence the expected number of sequences from $\mathcal{A}^\varepsilon_N(w_0)$ that appear as sequences of first entries of $(\alpha_1, \alpha_2, \ldots)$ is exactly $N+1$. On the other hand, due to Lemmata \ref{lemma:prbound} and \ref{lemma:meanbound}, the infimum of their expected number over all $\varepsilon : E \to \mathbb{R}_+$ is at most
	$$
		\frac{1}{\Pr(w_0.X\in\mathcal{G})}\left( \inf_{\varepsilon : E \to \mathbb{R}_+} \sum_{\mathcal{A}^\varepsilon_N(w)} P(w_0, \ldots, w_n; \alpha_1, \ldots, \alpha_n)\right)\, \leq\, \frac{\func(\beta)}{\Pr(w_0.X\in\mathcal{G})},
	$$
	i.e. it is bounded by some constant that is independent of $N$. This contradiction completes the proof of Theorem \ref{theo:LAL}.
			
	\section{Applications}\label{sec:applications}
	
	\subsection{Proper vertex colorings}
	
	Let us start with proving the following ``toy'' theorem.
	
	\begin{theo}\label{theo:propcol}
		Every graph\footnote{All graphs considered here are finite, undirected, and simple.} $G(V,E)$ admits a proper vertex coloring by $\Delta(G)+1$ colors. 
	\end{theo}
	
	We are using this very easy example to show how combinatorial statements can be transformed into instances for the LAL. Note that the LLL applied to this problem gives only a bound of about $e\Delta$ instead of $\Delta + 1$.
	
	\begin{proof}
		Many combinatorial problems (we will discuss several of them later) may be turned into special cases of the LAL in the following manner. If $A$ and $B$ are sets, then a \emph{partial function} from $A$ to $B$ (notation: $f : A \dashrightarrow B$) is a map $f : A' \to B$ for some $A' \subseteq A$. Denote the set of all partial finctions from $A$ to $B$ by $\mathsf{PF}(A, B)$. For $f \in \mathsf{PF}(A,B)$ let $\dom(f) \subseteq A$ be the domain of $f$. If $\dom(f) = A$, then to emphasize that fact we would sometimes call $f$ a \emph{total function} from $A$ to $B$. If $A$ is finite and $B$ is at most countable, then $\mathsf{PF}(A, B)$ is at most countable, so we may equip it with the $\sigma$-algebra $\Sigma_{A, B} \coloneqq \mathcal{P}(\mathsf{PF}(A, B))$ of all its subsets. The power set $\mathcal{P}(A)$ can be turned into a commutative monoid with the multiplication given by the union operation that acts on $\mathsf{PF}(A, B)$ with
		$$S.f = f|_{\dom(f) \setminus S}.$$
		Clearly, this action is measurable with respect to $\Sigma_{A, B}$. Observe that if $A$ is finite, then as a monoid $\mathcal{P}(A)$ is generated by the set $\{\{a\}\,:\,a\in A\}$ of singletons. Slightly abusing the notation, we would indentify the singletons with the corresponding elements of $A$. In particular, we would say that $\mathcal{P}(A)$ is generated by $A$, and write $a.f$ instead of $\{a\}.f$.
		
		Consider a graph $G(V, E)$ with maximum degree $\Delta$. Then $\mathsf{PF}(V, [\Delta+1])$ is the set of all partial colorings of $G$. Let $\mathcal{G} \subseteq \mathsf{PF}(V, [\Delta+1])$ be the set of all proper partial colorings. Note that for every $S \in \mathcal{P}(V)$ and $f \in \mathcal{G}$ we have $S.f \in \mathcal{G}$ (any restriction of a proper coloring is proper). Now choose a total coloring of $G$ uniformly at random. That gives us a random variable $X \in \mathsf{PF}(V, [\Delta+1])$. We want to prove that $\Pr(X \in \mathcal{G}) > 0$. Note that $V.X = \emptyset \in \mathcal{G}$, in particular, $\Pr(V.X \in \mathcal{G}) = 1 > 0$.
		
		Suppose that $f \not\in \mathcal{G}$, but $v.f \in \mathcal{G}$ for some $v \in V$. Let $g_v(f) \coloneqq \emptyset$. Then for $v \in V$ and $S \in \mathcal{P}(V)$ we have $P(v, S) \neq 0$ only if $S = \emptyset$. In the latter case
		$$P(v, \emptyset) = \sup_{S \subseteq V}
			\Pr\left(X|_{V\setminus S} \not \in \mathcal{G}\,\middle\vert\, X|_{V\setminus(S\cup\{v\})} \in \mathcal{G}\right).$$
		In other words, we have to give an upper bound for the probability that a random coloring of $V \setminus S$ is improper given that its restriction to $V \setminus (S \cup\{v\})$ is proper. Since there can be no more than $\Delta$ forbidden colors for $v$ in the restricted coloring, this probability is at most $\Delta/(\Delta+1)$. So it remains to find a function $\func : V \to \mathbb{R}_+$ such that for every $v \in V$ we have
		$$\func(v) \geq 1 + \frac{\Delta}{\Delta + 1} \func(v).$$
		Setting $\func(v) = \Delta + 1$ for all $v \in V$ completes the proof.
	\end{proof}
	
	\subsection{The LAL implies the Lopsided LLL}\label{subsec:LLL}
	
	Now we are going to show the connection between the LAL and the LLL. Namely, we will prove that the Lopsided LLL, which is a strengthening of the LLL, may be derived as a particular case of the LAL.
	
	\begin{theo}[Lopsided Lov\'{a}sz Local Lemma, \cite{Erdos}]\label{theo:LopLLL}
		Let $\beauty{A}$ be a finite set of random events in a probability space $\Omega$. For $A\in\beauty{A}$ let $\Gamma(A)$ be a subset of $\beauty{A}\setminus\{A\}$ such that
		$$\Pr(A) \geq \Pr\left(A\middle\vert \bigcap_{B \in \beauty{S}} \overline{B} \right)$$
		for every $\beauty{S} \subseteq \beauty{A}\setminus(\Gamma(A)\cup\{A\})$. Suppose that there exists an assignment of reals $\mu:\beauty{A}\to[0;1)$ such that for every $A\in\beauty{A}$ we have
		\begin{equation}\label{eq:LopLLL}
			\Pr(A)\leq \mu(A) \prod_{B\in \Gamma(A)}(1-\mu(B)).
		\end{equation}
		Then $\Pr\left(\bigcap_{A\in\beauty{A}}\overline{A}\right) > 0$.
	\end{theo}
	\begin{proof}
		Consider the set $\mathsf{PF}(\beauty{A}, \{0, 1\})$. Let $\mathcal{G} \subseteq \mathsf{PF}(\beauty{A}, \{0, 1\})$ be the set of partial functions that do not take the value $1$. For $\omega \in \Omega$ let $X(\omega)$ be the total function from $\beauty{A}$ to $\{0, 1\}$ that is equal to $1$ if and only if $\omega \in A$. Our goal then is to show that $X \in \mathcal{G}$ with positive probability.
		
		Note that $\beauty{A}.X = \emptyset \in \mathcal{G}$ with probability $1$. If $f \not\in \mathcal{G}$, but $A.f \in \mathcal{G}$ for some $A \in \beauty{A}$, then let $g_A(f) \coloneqq \Gamma(A)$. Then $P(A, \beauty{S}) \neq 0$ only if $\beauty{S} = \Gamma(A)$. In the latter case
		$$P(A, \Gamma(A)) = \sup_{\beauty{S} \subseteq \beauty{A}} \Pr\left( X|_{\beauty{A}\setminus\beauty{S}} \not\in \mathcal{G} \text{ and } X|_{\beauty{A}\setminus(\beauty{S}\cup\{A\})}\in\mathcal{G}\;\middle\vert\; X|_{\beauty{A}\setminus(\beauty{S} \cup \Gamma(A)\cup\{A\})} \in\mathcal{G} \right).$$
		Clearly, if $A \in \beauty{S}$, then
		$$\Pr\left( X|_{\beauty{A}\setminus\beauty{S}} \not\in \mathcal{G} \text{ and } X|_{\beauty{A}\setminus(\beauty{S}\cup\{A\})}\in\mathcal{G}\;\middle\vert\; X|_{\beauty{A}\setminus(\beauty{S} \cup \Gamma(A)\cup\{A\})} \in\mathcal{G} \right) = 0.$$
		If, on the other hand, $A \not \in \beauty{S}$, then
		$$\Pr\left( X|_{\beauty{A}\setminus\beauty{S}} \not\in \mathcal{G} \text{ and } X|_{\beauty{A}\setminus(\beauty{S}\cup\{A\})}\in\mathcal{G}\;\middle\vert\; X|_{\beauty{A}\setminus(\beauty{S} \cup \Gamma(A)\cup\{A\})} \in\mathcal{G} \right) \leq \Pr\left(A\,\middle\vert\,\bigcap_{B \in \beauty{A}\setminus(\beauty{S} \cup \Gamma(A)\cup\{A\})} \overline{B}\right) \leq \Pr(A).$$
		
		Thus we need to find a function $\func : \beauty{A} \to \mathbb{R}_+$ satisfying
		$$\func(A) \geq 1 + \Pr(A) \prod_{B \in \Gamma(A)\cup\{A\}} \func(B)$$
		for all $A \in \beauty{A}$. Take $\func(A) \coloneqq 1/(1-\mu(A))$. Then
			\begin{align*}
			1 + \Pr(A) \prod_{B \in \Gamma(A)\cup\{A\}} \func(B) &= 1 + \frac{\Pr(A)}{\prod_{B \in \Gamma(A)\cup\{A\}} (1-\mu(B))} \\
			& \leq 1 + \frac{\mu(A) \prod_{B\in \Gamma(A)}(1-\mu(B))}{\prod_{B \in \Gamma(A)\cup\{A\}} (1-\mu(B))} \\
			& = 1 + \frac{\mu(A)}{1-\mu(A)} = \frac{1}{1-\mu(A)} = \func(A),
			\end{align*}	
		and we are done.
	\end{proof}
	
	\subsection{A lower bound on the probability}
	
	One of the main ideas of the entropy compression method is to avoid bounding the probability explicitly, but instead focus on the expected runtime of a certain randomized algorithm. So it is unclear that the entropy compression can possibly produce any explicit lower bound on the probability, except for just showing that it must be positive. Note that the classical LLL, on the other hand, does give a lower bound on the probability, namely it asserts that $\Pr\left(\bigcap_{A\in\beauty{A}}\overline{A}\right) \geq \prod_{A\in\beauty{A}} (1 - \mu(A))$. Although the proof of the LAL, following the philosophy of the entropy compression method, does not bound the probability explicitly, the statement itself, surprisingly enough, can be bootstrapped to give an explicit lower bound as a corollary (and this lower bound implies the standard lower bound of the LLL). To obtain this bound we exploit the fact that the LAL is stated in a very general and abstract way. That allows us to construct an auxiliary instance for the LAL, applying the lemma to which gives the desired result.
	
	\begin{theo}\label{theo:LB}
		In the setting of the LAL, for every $\alpha \in M$ we have
		$$
			\Pr(X \in \mathcal{G}) \geq \frac{\Pr(\alpha.X \in \mathcal{G})}{\underline{\func}(\alpha) }.
		$$
	\end{theo}
	\begin{proof}
		Construct an auxiliary instance for the LAL in the following manner. Let $\mathcal{X}'$ be a copy of $\mathcal{X}$ disjoint from $\mathcal{X}$, and for $x \in X$ denote its copy in $\mathcal{X}'$ by $x'$. Let $M'$ be obtained from $M$ by adding one new generator $\nu$ and declaring that $\nu\alpha = \alpha$ for every $\alpha \in M \setminus \{1\}$, and $\nu^2 = \nu$. So $M' = M \cup \{\alpha\nu\,:\,\alpha \in M\}$ as a set. Let $B' \coloneqq B \cup \{\nu\}$ be a generating set for $M'$. Let $M'$ act on $\mathcal{X} \cup \mathcal{X}'$ in the following way: If $\alpha \in M$, $x \in \mathcal{X}$, then let $\alpha.x$ be the same as in the action of $M$ on $\mathcal{X}$; if $\alpha \in M$ and $x' \in \mathcal{X}'$, then let $\alpha.x' = x'$; if  $x \in \mathcal{X}$, then let $\nu.x = x$; finally, if $x' \in \mathcal{X}'$, then let $\nu.x' = x$. In other words, we allow $M$ to act on $\mathcal{X}$ as before, whereas $\nu$ send the elements from $\mathcal{X}'$ back to their respective copies in $\mathcal{X}$.
		
		Let the set $\mathcal{G} \subseteq \mathcal{X}$ stay the same as before (so $\mathcal{G} \cap \mathcal{X}' = \emptyset$). To define the random variable $X' \in \mathcal{X}\cup \mathcal{X}'$, pick a random element from $\mathcal{X}'$, using the distribution of $X \in \mathcal{X}$ (so the notation is correct: $X'$ is the variable obtained from $X$ by sending it to $\mathcal{X}'$ via the map $' : \mathcal{X} \to \mathcal{X}'$).
		
		If $x \in \mathcal{X} \setminus \mathcal{G}$, but $\beta.x \in \mathcal{G}$ for some $\beta \in B'$, then $\beta \in B$, since $\nu.x = x$, so we can define $g_\beta(x)$ to be the same as before. If, on the other hand, $x' \in \mathcal{X}'$ and $\beta.x' \in \mathcal{G}$, then $\beta = \nu$, since otherwise $\beta.x' = x'$. In this case let $g_\nu(x') \coloneqq \alpha$ for some fixed $\alpha \in M$.
		
		Clearly, since $M'$ never sends any elements from $\mathcal{X}$ to $\mathcal{X}'$, the inequalities of the LAL corresponding to $\beta \in B$ are exactly the same as before, and by assumption they are satisfied for $\func: B \to \mathbb{R}_+$. But, by definition, $X' \in \mathcal{X}'$, so $\Pr(X' \in \mathcal{G}) = 0$. Thus the inequality corresponding to $\nu$ must be violated. This inequality is
		$$
			\func(\nu) \geq 1 + P(\nu, \alpha) \underline{\func}(\alpha\nu).
		$$
		Note that $\underline{\func}(\alpha\nu) = \underline{\func}(\alpha) \func(\nu)$, so we can rewrite the last inequality as
		$$
			\func(\nu) \geq 1 + P(\nu, \alpha)\underline{\func}(\alpha) \func(\nu).
		$$
		By definition,
		$$
			P(\nu, \alpha) = \sup_{\gamma \in M'} \Pr\left( g_\nu(\gamma.X') = \alpha \,\middle\vert\, \alpha\nu\gamma.X' \in \mathcal{G}\right).
		$$
		If $\gamma \in M' \setminus M$, then $\gamma.X' \in \mathcal{X}$, and hence $g_\nu(\gamma.X')$ is undefined. But if $\gamma \in M$, then $\gamma.X' = X'$, so
		$$
			P(\nu, \alpha) = \Pr\left(g_\nu(X') = \alpha\,\middle\vert\, \alpha\nu.X' \in \mathcal{G}\right).
		$$
		Since $\nu.X' = X$, and $g_\nu(X') = \alpha$ if and only if $\nu.X' \in \mathcal{G}$, we get
		$$
			P(\nu, \alpha) = \Pr\left(X \in \mathcal{G}\,\middle\vert\,\alpha.X \in \mathcal{G}\right) = \frac{\Pr(X \in \mathcal{G})}{\Pr(\alpha.X \in \mathcal{G})}.
		$$
		Thus the following inequality does not hold for any $\func(\nu) \in \mathbb{R}_+$:
		$$
			\func(\nu) \geq 1 + \frac{\Pr(X \in \mathcal{G})}{\Pr(\alpha.X \in \mathcal{G})}\underline{\func}(\alpha) \func(\nu).
		$$
		But this inequality does not have a solution $\func(\nu) \in \mathbb{R}_+$ if and only if
		$$
			\frac{\Pr(X \in \mathcal{G})}{\Pr(\alpha.X \in \mathcal{G})}\underline{\func}(\alpha) \geq 1,
		$$
		and we are done.
	\end{proof}
	
	In the framework discussed in Section \ref{subsec:LLL}, we have
	$$
		\Pr(\beauty{A}.X \in \mathcal{G}) = 1,
	$$
	and
	$$
		\underline{\func}(\beauty{A}) = \prod_{A \in \beauty{A}} \func(A) = \frac{1}{\prod_{A \in \beauty{A}}(1 - \mu(A))},
	$$
	so applying Theorem \ref{theo:LB} gives
	$$
		\Pr\left(\bigcap_{A\in\beauty{A}}\overline{A}\right) = \Pr(X \in \mathcal{G}) \geq \frac{\Pr(\alpha.X \in \mathcal{G})}{\underline{\func}(\alpha) } = \prod_{A \in \beauty{A}}(1 - \mu(A)),
	$$
	as desired.
	
	\subsection{Non-repetitive sequences and non-repetitive colorings}
	
		The study of non-repetitive sequences and non-repetitive colorings is one of the first areas of combinatorics where the entropy compression method was successfully applied. Here we are going to present two of these results and reprove them using the LAL.
	
		A fnite sequence $a_1 a_2 \ldots a_n$ is \emph{non-repetitive}, if there are no $1 \leq t \leq \left\lfloor n/2\right\rfloor$ and $1\leq s \leq n-2t+1$ such that $a_k = a_{k+t}$ for all $s\leq k \leq s+t-1$. A well known result by Thue \cite{Thue} asserts that there exist arbitrarily long non-repetitive sequences of elements from $\{0, 1, 2\}$. The following theorem is a choosability version of this result.
		
		\begin{theo}[Grytczuk \emph{et al.} \cite{Grytczuk}]\label{theo:nonrepseq}
			Let $L_1$, $L_2$, \ldots, $L_n$ be a sequence of sets with $|L_k| \geq 4$ for all $1\leq k \leq n$. Then there exists a non-repetitive sequence $a_1 a_2 \ldots a_n$ such that $a_k \in L_k$ for all $1 \leq k \leq n$.
		\end{theo}
		\begin{remk}
			It is an open problem whether the same result is true for $|L_k| \geq 3$.
		\end{remk}
		\begin{proof}
			Let $\mathcal{X}$ be the set of sequences $a_1 a_2 \ldots a_m$ with $0 \leq m \leq n$ such that $a_k \in L_k$ for all $1\leq k \leq m$, and let $\mathcal{G} \subseteq \mathcal{X}$ be the subset of all non-repetitive sequences. Let $\{\beta\}^*$ be the free monoid with one generator. Then $\{\beta\}^*$ acts on $\mathcal{X}$ with
			$\beta. a_1 a_2 \ldots a_m = a_1 a_2 \ldots a_{m-1}$
			for $m > 0$ and $\beta.\emptyset = \emptyset$. This action is, of course, measurable with respect to the $\sigma$-algebra $\Sigma \coloneqq \mathcal{P}(\mathcal{X})$. Choose a sequence $X = X_1 X_2 \ldots X_n$ of length $n$ from $\mathcal{X}$ uniformly at random. Then $\Pr(\beta^n.X \in \mathcal{G}) = 1 >0$. We want to show that $\Pr(X \in \mathcal{G}) > 0$.
			
			Suppose that $a_1 a_2 \ldots a_m \not \in \mathcal{G}$, but $a_1 a_2 \ldots a_{m-1} \in \mathcal{G}$. Then there is $1 \leq t \leq \left\lfloor m/2\right\rfloor$ such that $a_k = a_{k+t}$ for all $m-2t+1\leq k \leq m-t$. Let $g_\beta(a_1 a_2 \ldots a_m) \coloneqq \beta^{t-1}$ (if there is more than one such $t$, choose any one of them).
			
			What is $P(\beta, \beta^{t-1})$? We have
			\begin{align*}
				P(\beta, \beta^{t-1}) &= \sup_{\ell \geq 0} \Pr\left( g_\beta(\beta^\ell.X) = \beta^{t-1} \,\middle\vert\, \beta^{\ell+t}.X \in \mathcal{G} \right)\\
			&\leq \sup_{\ell \geq 0} \Pr\left( X_k = X_{k+t} \text{ for all } n - \ell - 2t + 1 \leq k \leq n - \ell - t \,\middle\vert\, X_1 \ldots X_{n-\ell - t} \in \mathcal{G} \right)
			\leq \frac{1}{4^t}.
			\end{align*}
			Hence it is enough to find a number $\func \in \mathbb{R}_+$ such that
			$$\func \geq 1 + \sum_{t = 1}^\infty \frac{\func^t}{4^t} = \frac{1}{1-\func/4},$$
			where the last equality is subject to $\func < 4$. Setting $\func = 2$ completes the proof.
		\end{proof}
	
		A vertex coloring $f$ of a graph $G(V,E)$ is \emph{non-repetitive} if there is no path $P$ in $G$ with even number of vertices (we call it the \emph{length} of a path) such that the first half of $P$ receives the same sequence of colors as the second half of $P$, i.e. if there is no path $v_1$, $v_2$, \ldots, $v_{2t}$ such that $f(v_k) = f(v_{k+t})$ for all $1\leq k \leq t$. The least number of colors that is needed for a non-repetitive coloring of $G$ is called the \emph{non-repetitive chromatic number} of $G$ and is denoted by $\pi(G)$.
		
		The first upper bound for $\pi(G)$ in terms of the maximum degree $\Delta(G)$ was given by Alon \emph{et al.} \cite{Alon3}, who proved that there is a constant $c$ such that $\pi(G)\leq c \Delta(G)^2$. Originally this result was obtained with $c = 2e^{16}$. The constant was then improved to $c = 16$ by Grytczuk \cite{Grytczuk1}, and then to $c = 12.92$ by Haranta and Jendrol' \cite{Haranta}. All these results were based on the LLL.
		
	Dujmovi\'{c} \emph{et al.} \cite{Duj} managed to improve the value of the aforementioned constant $c$ dramatically using the entropy compression method. Namely, they lowered the constant to $1$, or, to be precise, they showed that $\pi(G)\leq(1+o(1))\Delta(G)^2$ (assuming that $\Delta(G)\to\infty$).
		
	The best known bound is given by the following theorem.
	
	\begin{theo}[Gon\c{c}alves \emph{et al.} \cite{Goncalves}]
		For every graph $G(V, E)$ with maximum degree $\Delta$ we have
			$$\pi(G) \leq \left\lceil \Delta^2 + \frac{3}{2^{2/3}}\Delta^{5/3} +\frac{2^{2/3} \Delta^{5/3}}{\Delta^{1/3}-2^{1/3}} \right\rceil.$$
	\end{theo}
	\begin{proof}
		Our framework will be the same as in the proofs of Theorems \ref{theo:propcol} and \ref{theo:LopLLL}. Let $G(V, E)$ be a graph with maximum degree $\Delta$, and let $C$ be a set of colors of cardinality
		\begin{equation}\label{eq:non-repbound}
			|C| \geq \Delta^2 + \frac{3}{2^{2/3}}\Delta^{5/3} +\frac{2^{2/3} \Delta^{5/3}}{\Delta^{1/3}-2^{1/3}}.
		\end{equation}
		Let $\mathcal{G} \subseteq \mathsf{PF}(V, C)$ be the set of all non-repetitive partial colorings of $G$. Choose a total coloring $X \in \mathsf{PF}(V, C)$ uniformly at random. Note that $\Pr(V.X \in \mathcal{G}) = 1 >0$, and we want to prove that $\Pr(X \in \mathcal{G}) > 0$.
		
		If $f : V \dashrightarrow C$ and $f \not \in \mathcal{G}$, but $v.f \in \mathcal{G}$ for some $v \in V$, then there exists an even path $P$ contained in $\dom(f)$ that passes through $v$ and is colored repetitively by $f$. Let $g_v(f) \coloneqq V(P'_v)$, where $P'_v$ is the half of $P$ that contains $v$ (if there is more than one such path $P$, choose any one of them).
		
		Let $v \in V$ and $S \in \mathcal{P}(V)$. Then $P(v, S) \neq 0$ only if $S = V(P'_v)$ for a path $P$ of length $2t$ containing~$v$, where $t = |S|$. In the latter case let $\Pi(S)$ be the set of all paths $P$ of length $2t$ such that $S = V(P'_v)$. Then
		\begin{align*}
			P(v, S) & = \sup_{U \subseteq V} \Pr\left( g_v(X|_{V \setminus U}) = S \,\middle\vert\, X|_{V \setminus (U \cup S)} \in \mathcal{G} \right)\\
			&\leq \sup_{U \subseteq V} \sum_{P \in \Pi(S)} \Pr\left( P \text{ is colored repetitively in } X|_{V \setminus U} \,\middle\vert\, X|_{V \setminus (U \cup V(P'_v))} \in \mathcal{G} \right)\\
			&\leq \sum_{P \in \Pi(S)} \frac{1}{|C|^t}.
		\end{align*}
		
		Note that there are at most $t \Delta^{2t -1}$ paths of length $2t$ containing a given vertex $v \in V$. Indeed, there are $2t$ ways of choosing a position for $v$ on a path of length $2t$ and at most $\Delta^{2t -1}$ ways of choosing a path of length $2t$ passing through $v$ assumming that the position of $v$ on a path is fixed, but in this manner we count every path twice.
		Finally, suppose that $\func(v) = \func \in \mathbb{R}_+$ is a costant independent of $v \in V$. Then it would suffice to have
		\begin{equation}\label{eq:non-rep}
			\func \geq 1 + \sum_{t = 1}^\infty \frac{t \Delta^{2t - 1} \func^t}{|C|^t} = 1 + \frac{\Delta \func/|C|}{(1-\Delta^2 \func/|C|)^2},
		\end{equation}
		where the last equality is subject to $\Delta^2 \func/|C| < 1$. If we let $y \coloneqq \Delta^2 \func/|C|$, then (\ref{eq:non-rep}) turns into
		\begin{equation}\label{eq:non-rep2}
			\frac{|C|}{\Delta^2} \geq \frac{1}{y} + \frac{1}{\Delta(1-y)^2}.
		\end{equation}
		Following Gon\c{c}alves \emph{et al.}, we take $y = 1 - \left(2/\Delta\right)^{1/3}$, and (\ref{eq:non-rep2}) becomes
		$$
			\frac{|C|}{\Delta^2} \geq 1+ \frac{3}{2^{2/3}\Delta^{1/3}} +\frac{2^{2/3}}{\Delta^{2/3}-(2\Delta)^{1/3}}, 
		$$
		which is true by (\ref{eq:non-repbound}).
	\end{proof}
	
	\subsection{Acyclic edge colorings}\label{subsec:acyclic}
	
	An edge coloring of a graph $G$ is called an \emph{acyclic edge coloring} if it is proper (i.e. adjacent edges receive different colors) and every cycle in $G$ contains edges of at least three different colors (there are no \emph{bichromatic cycles} in $G$). The least number of colors needed for an acyclic edge coloring of $G$ is called the \emph{acyclic chromatic index} of $G$ and is denoted by $a'(G)$. The notion of acyclic (vertex) coloring was first introduced by Gr\"{u}nbaum \cite{Grunbaum}. The edge version was first considered by Fiam\v{c}ik \cite{Fiamcik}, and independently by Alon \emph{et al.} \cite{Alon1}.
	
	As in the case of non-repetitive colorings, it is quite natural to ask for an upper bound on the acyclic chromatic index of a graph $G$ in terms of its maximum degree $\Delta(G)$. Since $a'(G)\geq \func'(G) \geq \Delta(G)$, this bound must be at least linear in $\Delta(G)$. The first linear bound was given by Alon \emph{et al.} \cite{Alon1}, who showed that $a'(G)\leq 64 \Delta(G)$. Although it resolved the problem of determining the order of growth of $a'(G)$ in terms of $\Delta(G)$, it was conjectured that the sharp bound should be much lower.
	
	\begin{conj}[Fiam\v{c}ik \cite{Fiamcik}, Alon \emph{et. al.} \cite{Alon2}]\label{conj:AECC}
		For every graph $G$ we have $a'(G) \leq \Delta(G)+2$.
	\end{conj}
	
	Note that the bound in Conjecture \ref{conj:AECC} is only one more than Vizing's bound on the chromatic index of $G$. However, this elegant conjecture is still far from being proven.
	
	The first major improvement to the bound $a'(G)\leq 64 \Delta(G)$ was made by Molloy and Reed \cite{Molloy}, who proved that $a'(G) \leq 16 \Delta(G)$. This bound remained the best for a while, until Ndreca \emph{et al.} \cite{Ndreca} managed to improve it to $a'(G)\leq \left\lceil 9.62(\Delta(G)-1)\right\rceil$. Again, first bounds for $a'(G)$ were obtained using the LLL. The bound $a'(G)\leq \left\lceil 9.62(\Delta(G)-1)\right\rceil$ by Ndreca \emph{et al.} used an improved version of the LLL due to Bissacot \emph{et al.} \cite{Bissacot}.
	
	The best current bound for $a'(G)$ in terms of $\Delta(G)$ was obtained recently by Esperet and Parreau via the entropy compression method.
	\begin{theo}[Esperet \& Parreau \cite{Esperet}]\label{theo:acyclic}
		For every graph $G(V, E)$ with maximum degree $\Delta$ we have $a'(G) \leq 4(\Delta-1)$.
	\end{theo}
	Below we give two LAL-based proofs of Theorem \ref{theo:acyclic}, that differ in their ways of ``translating'' the algorithmic proof by Esperet and Parreau into the language of probability.
	
	\begin{proof}[First proof]
		We start with the following claim, that plays a crucial role in our proof as well as in the original proof by Esperet and Parreau.
		\begin{claim}\label{claim:acyclicext}
			Let $G(V, E)$ be a graph with maximum degree $\Delta$. Suppose that some of its edges are colored properly by $k > 2\Delta$ colors, and suppose that an edge $e \in E$ is uncolored. Then there exist at least $k - 2(\Delta-1)$ ways to color $e$ so that the resulting coloring is still proper and does not contain bichromatic $4$-cycles going through $e$.
		\end{claim}
		Indeed, let $e = uv$. If coloring $e$ with a color $c$ creates a bichromatic $4$-cycle $uvxy$, then the edges $vx$ and $uy$ must be already colored the same. On the other hand, any pair of edges $vx$ and $uy$ that are colored the same can give rise to at most one bichromatic $4$-cycle, namely $uvxy$. The number of edges adjacent to $e$ is at most $2(\Delta-1)$, and to each color that creates a bichromatic $4$-cycle corresponds exactly one pair of edges adjacent to $e$ that share their color. Therefore, the number of ``forbidden'' colors for $e$ is at most $2(\Delta - 1)$, which establishes the claim.
		
		The crucial idea of \cite{Esperet} (which is credited to Jakub Kozik by the authors) is to handle $4$-cycles and cycles of length at least $6$ separately. It can be done in the probabilistic framework in the following way. Let $C$ be a set of colors of cardinality $|C| \geq 4(\Delta-1)$. Claim \ref{claim:acyclicext} ensures then that there exist proper edge colorings of $G$ by colors $C$ which do not contain bichromatic $4$-cycles. Let $\mathcal{X} \subseteq \mathsf{PF}(E, C)$ be the set of all such partial colorings, and let $\mathcal{G} \subseteq \mathcal{X}$ be the set of all acyclic partial colorings. Choose a total coloring $X \in \mathcal{X}$ uniformly at random. We want to show that $\Pr(X \in \mathcal{G}) > 0$.
		
		If $f \in \mathcal{X}\setminus \mathcal{G}$ and $e.f \in \mathcal{G}$ for some $e \in E$, then there is a cycle $K$ of length at least $6$ contained in $\dom(f)$ that passes through $e$ and is colored bichromatically by $f$. Let $g_e(f)\coloneqq K'_e$, where $K'_e$ is a subset of $E(K)$ that contains $e$ and all other edges of $K$ except for two arbitrary adjacent edges (if there is more than one such $K$, choose any one of them).
		
		Now $P(e, S) \neq 0$ only if $S = K'_e$ for some cycle $K$ of length $2t\geq 6$, where $2t = |S|+2$. In the latter case let $\Theta(S)$ be the set of all cycles $K$ such that $K'_e = S$. Then
		\begin{align*}
			P(e, S) & = \sup_{F \subseteq E} \Pr\left( g_e(X|_{E \setminus F}) = S \,\middle\vert\, X|_{E \setminus (F \cup S)} \in \mathcal{G} \right)\\
			&\leq \sup_{F \subseteq E} \sum_{K \in \Theta(S)} \Pr\left( K \text{ is colored bichromatically in } X|_{E \setminus F} \,\middle\vert\, X|_{E \setminus (F \cup K'_e)} \in \mathcal{G} \right).
		\end{align*}
		Note that the probability
		$$\Pr\left( K \text{ is colored bichromatically in } X|_{E \setminus F} \,\middle\vert\, X|_{E \setminus (F \cup K'_e)} \in \mathcal{G} \right)$$
		is non-zero only if $E(K) \subseteq E \setminus F$. In that case the set $E \setminus (F \cup K'_e)$ contains exactly two adjacent edges of $K$. If $f : E \setminus K'_e \to C$ is an acyclic coloring, then due to Claim \ref{claim:acyclicext} it can be extended to the whole set $E$ properly and without bichromatic $4$-cycles in at least $(2(\Delta - 1))^{2t - 2}$ ways (recall that $|C| \geq 4(\Delta-1)$), but $K$ is bichromatic in at most one of them. Therefore,
		\begin{align*}
			\sup_{F \subseteq E} &\sum_{K \in \Theta(S)} \Pr\left( K \text{ is colored bichromatically in } X|_{E \setminus F} \,\middle\vert\, X|_{E \setminus (F \cup K'_e)} \in \mathcal{G} \right)\\
			&\leq \sum_{K \in \Theta(S)} \frac{1}{(2(\Delta - 1))^{2t - 2}}.
		\end{align*}
		
		Since there are at most $(\Delta-1)^{2t-2}$ cycles of length $2t$ passing through a given edge $e$, it is enough to find a number $\func \in \mathbb{R}_+$ such that
	\begin{equation}\label{eq:acyc}
		\func \geq 1 + \sum_{t = 3}^\infty \frac{(\Delta-1)^{2t-2} \func^{2t-2}}{(2(\Delta - 1))^{2t - 2}} = 1 + \frac{\left(\func/2\right)^4}{1-\left(\func/2\right)^2},
	\end{equation}
	where the last equality is subject to $\func/2 < 1$. Setting $\func = \sqrt{5}-1$ completes the proof.
	\end{proof}
	
	\begin{proof}[Second proof]
		In the previous proof the set $\mathcal{X}$ was constructed in a clever and nontrivial way. It turns out, however, that the LAL is flexible enough to produce the same result in the case where the random coloring is just chosen uniformly at random from $C^E$.
		
		Let $\mathcal{G}\subseteq \mathsf{PF}(E, C)$ be the set of all acyclic partial colorings, and let $X$ be a total coloring chosen from $\mathsf{PF}(E, C)$ uniformly at random. If $f : E \dashrightarrow C$ and $f \not \in \mathcal{G}$, but $e.f \in \mathcal{G}$ for some $e \in E$, then one of the following holds.
		\begin{enumerate}
			\item There is an edge $h \in \dom(f)$ that is adjacent to $e$ and $f(e) = f(h)$. In this case let $g_e(f) \coloneqq \emptyset$.
			\item There is a $4$-cycle contained in $\dom(f)$ that passes through $e$ and is colored bichromatically by~$f$. In this case let $g_e(f) \coloneqq \emptyset$.
			\item There is a cycle $K$ of length at least $6$ contained in $\dom(f)$ that passes through $e$ and is colored bichromatically by $f$. In this case let $g_e(f)\coloneqq K'_e$, where $K'_e$ is defined in the same way as in the previous proof.
		\end{enumerate}
		Again, if there is some ambiguity in the definition of $g_e(f)$, then choose any available option.
		
		Now $P(e, S) \neq 0$ only if $S = \emptyset$ or $S = K'_e$ for some cycle $K$ of length $2t\geq 6$, where $2t = |S|+2$. In the first case
		\begin{align*}
			P(e, \emptyset) = \sup_{F \subseteq E} \Pr\left( g_e(X|_{E \setminus F}) = \emptyset \,\middle\vert\, X|_{E \setminus (F \cup \{e\})}\in\mathcal{G} \right) \,\leq\, \frac{2(\Delta - 1)}{4(\Delta - 1)} = \frac{1}{2},
		\end{align*}
		where the inequality is due to Claim \ref{claim:acyclicext}. In the second case let $\Theta(S)$ be the set of all cycles $K$ such that $K'_e = S$. Then, analogously to the first proof,
		\begin{align*}
			P(e, S) & = \sup_{F \subseteq E} \Pr\left( g_e(X|_{E \setminus F}) = S \,\middle\vert\, X|_{E \setminus (F \cup S)} \in \mathcal{G} \right)\\
			&\leq \sup_{F \subseteq E} \sum_{K \in \Theta(S)} \Pr\left( K \text{ is colored bichromatically in } X|_{E \setminus F} \,\middle\vert\, X|_{E \setminus (F \cup K'_e)} \in \mathcal{G} \right)\\
			&\leq \sum_{K \in \Theta(S)} \frac{1}{(4(\Delta - 1))^{2t - 2}}.
		\end{align*}
		Finally, it is enough to find a number $\func \in \mathbb{R}_+$ such that
		\begin{equation}\label{eq:acyc1}
			\func \geq 1 + \sum_{t = 3}^\infty \frac{(\Delta-1)^{2t-2} \func^{2t-2}}{(4(\Delta - 1))^{2t-2}} + \frac{\func}{2} = 1 + \frac{\left(\func/4\right)^4}{1-\left(\func/4\right)^2}+\frac{\func}{2},
		\end{equation}
	under the assumption that $\func/4 < 1$. If we denote $y = \func/2$, then (\ref{eq:acyc1}) turns into
		\begin{align*}
			y \geq 1 + \frac{\left(y/2\right)^4}{1-\left(y/2\right)^2},
		\end{align*}
		which is the same as (\ref{eq:acyc}), so we can take $y = \sqrt{5}-1$.
	\end{proof}
	
	\subsection{Ramsey numbers}
	
	When there are many different types of ``bad'' events that must be avoided, the LLL results in many inequalities with many variables, because it must assign different numbers to different events. The LAL, on the other hand, assigns numbers to the generators of the monoid, which are usually ``homogeneous'' in some sense, even if there are many ``bad'' events to handle. For example, if we want to prove an upper bound for the acyclic chromatic index of a graph using the LLL, we have to introduce different variables for the events that happen when two adjacent edges are colored the same, and for the events corresponding to bichromatic cycles of different lengths. Then we have to write down the inequalities for each of these events, and solve them simultaneously. The resulting problem, though purely computational, is not very pleasant. On the other hand, as we saw in Section \ref{subsec:acyclic}, the LAL gives us in this case only one inequality with a single variable $\func$.
	
	Let us illustrate this advantage of the LAL with the following example.
	
	\begin{theo}[Erd\H{o}s; Spencer \cite{AS}]\label{theo:ramsey}
		There is a positive constant $c$ such that $R(3, k) \geq (c + o(1)) \left(\frac{k}{\log k}\right)^2$, where $R(\ell, k)$ denotes the (off-diagonal) Ramsey number.
	\end{theo}
	\begin{remk}
		The best current bound is $R(3, k) \geq c \frac{k^2}{\log k}$ due to Kim \cite{Kim}.
	\end{remk}
	Note that the LLL-based proof of Theorem \ref{theo:ramsey} by Spencer boils down to the following analytic problem: What is the largest possible value of $n$ for which there exist $0 \leq x$, $y$, $p < 1$ with
	$$
		\left\{
			\begin{array}{rcl}
								p^3 & \leq & x (1-x)^{3n} (1-y)^{n \choose k}; \\
								(1-p)^{n \choose 2} & \leq & y (1-x)^{k^2 n/2} (1-y)^{n \choose k}?
			\end{array} 
		\right.
	$$
	\begin{proof}
		Let us color the edges of $K_n$ red and blue randomly and independently, choosing the blue color for each edge with probability $p$. It gives us a random total function $X \in \mathsf{PF}(E(K_n), \{\text{red, blue}\})$. Let $\mathcal{G} \subseteq \mathsf{PF}(E(K_n), \{\text{red, blue}\})$ be the set of all partial colorings with no blue triangle or red $K_k$. We want to prove that $\Pr(X \in \mathcal{G}) > 0$ for a suitable choice of $p$ and $n < (c + o(1)) \left(\frac{k}{\log k}\right)^2$.
		
		If $f : E(K_n) \dashrightarrow \{\text{red, blue}\}$ is not in $\mathcal{G}$, but $e.f \in \mathcal{G}$ for some $e \in E(K_n)$, then one of the following holds.
		\begin{enumerate}
			\item There is a blue triangle $\Delta$ in $\dom(f)$ that contains the edge $e$. Then let $g_e(f) \coloneqq E(\Delta)$.
			\item There is a red subgraph $H \cong K_k$ in $\dom(f)$ that contains the edge $e$. Then let $g_e(f) \coloneqq E(H)$.
		\end{enumerate}
		As usual, if there is some ambiguity in the definition of $g_e(f)$, then choose any available option.
		
		Now $P(e, S) \neq 0$ only if $S = E(H)$ for a graph $H$ that contains $e$ and is isomorphic to either $K_3$ or $K_k$. In the first case,
		\begin{align*}
			P(e, E(\Delta)) = p^3,
		\end{align*}
		and in the second case,
		$$
			P(e, E(H)) = (1-p)^{k \choose 2}.
		$$
		Since an edge $e$ is contained in $n-2$ triangles and ${n-2 \choose k-2}$ copies of $K_k$, it is enough to find a number $\func \in \mathbb{R}_+$ such that
		\begin{equation}\label{eq:ramsey}
			\func \geq 1 + (n-2)p^3\func^3 + {n-2 \choose k-2} (1-p)^{k \choose 2} \func^{k \choose 2}.
		\end{equation}
		So we get one inequality with two variables ($p$ and $\func$). But if we denote $x \coloneqq p \func$ and $y \coloneqq (1-p)\func$, then (\ref{eq:ramsey}) turns into
		$$
			\left(x - (n-2)x^3\right) + \left(y - {n-2 \choose k-2} y^{k \choose 2}\right) \geq 1.
		$$
		Hence we can optimize
		$$
			x - (n-2)x^3
		$$
		and
		$$
			y - {n-2 \choose k-2} y^{k \choose 2}
		$$
		separately, which is an easy one-variable maximization problem. The rest of the proof follows.
	\end{proof}
 	
	\subsection{A non-combinatorial example}
	
	All examples of using the LAL that we discussed before were combinatorial in nature and closely realted to the LLL, which is not surprising, since the LAL was intended to be a generalization of the LLL. Nevertheless, the framework of the LAL is more general than the usual LLL framework. Here we give only one easy example of an application of the LAL that is far from usual type of results obtained using the Local Lemma.
	
	\begin{theo}\label{theo:array}
		Let $(a_{ij})_{i, j = 1}^\infty$ be a collection of non-negative real numbers such that $$0 < \sum_{i, j = 1}^\infty a_{ij} < \infty.$$ For $i$, $j \geq 1$ let
		$$b_{ij} \coloneqq \frac{a_{ij}}{\sum_{k = 1}^\infty \sum_{\ell = 1}^{i + j - 1} a_{k \ell}}.$$
		Here we assume that $0/0 = 0$. Then for every $x \in \mathbb{R}_+$ we have
		$$1 + \sum_{i = 1}^\infty \left(\sup_{j\geq 1} b_{ij}\right)x^i \,>\,x.$$
	\end{theo}
	\begin{corl}\label{corl:seq}
		Let $(a_j)_{j=1}^\infty$ be a sequence of non-negative real numbers such that $0 < \sum_{j=1}^\infty a_j < \infty$. Let $k$ be a positive integer. Then
		$$\sup_{j \geq 1} \left(\frac{a_j}{\sum_{i = 1}^{j+k-1} a_i}\right) \, > \, \frac{1}{ek}.$$
	\end{corl}
	\begin{proof}[Proof of Corollary \ref{corl:seq}]
		It is easy to check that if $k > 1$ and $s \in \mathbb{R}_+$, then for every $x \in \mathbb{R}_+$ we have
		$$1 + s x^k > x$$
		if and only if
		$$s > \frac{(k-1)^{k-1}}{k^k} > \frac{1}{ek}.$$
		
		Let
			$$c_{ij} \coloneqq
				\begin{cases}
					0 &\mbox{if } i \neq k; \\
					a_j & \mbox{otherwise}. \end{cases}$$
		Then by Theorem \ref{theo:array} applied to $(c_{ij})_{i, j = 1}^\infty$ for every $x \in \mathbb{R}_+$ we have
		$$1 + \left(\sup_{j \geq 1} \left(\frac{a_j}{\sum_{i = 1}^{j+k-1} a_i}\right)\right) x^k \,>\, x.$$
		Therefore,
		$$\sup_{j \geq 1} \left(\frac{a_j}{\sum_{i = 1}^{j+k-1} a_i}\right) \, > \, \frac{1}{ek},$$
		as desired.
	\end{proof}
	\begin{proof}[Proof of Theorem \ref{theo:array}]
		Upon replacing each $a_{ij}$ with $\left.a_{ij}\middle/\left(\sum_{i, j = 1}^\infty a_{ij}\right)\right.$ we may assume that
		$$\sum_{i, j = 1}^\infty a_{ij} = 1.$$
		Let $\mathbb{N}_0$ denote the set of non-negative integers, and $\mathbb{N}_+$ denote the set of positive integers. Let $a_{i, 0} \coloneqq 0$ for every $i \in \mathbb{N}_+$. Let $\{\beta\}^*$ be the free monoid with one generator. Then $\{\beta\}^*$ acts on $\mathbb{N}_+ \times \mathbb{N}_0$ with
		$$\beta. (i, j) \coloneqq \begin{cases}
					(i, j - 1) &\mbox{if } j > 0; \\
					(i, j) & \mbox{otherwise}. \end{cases}$$
		Let $\mathcal{G} \coloneqq \mathbb{N}_+ \times \{0\}$. Since $\sum_{(i, j) \in \mathbb{N}_+ \times \mathbb{N}_0} a_{ij} = 1$, we may consider a random variable $(I, J) \in \mathbb{N}_+ \times \mathbb{N}_0$ such that
		$$\Pr((I, J) = (i, j)) = a_{ij}.$$
		For $i \in \mathbb{N}_+$ let $g_\beta((i, 1)) \coloneqq i - 1$. Then for any $i \in \mathbb{N}_+$ we have
		$$
			P(\beta, \beta^{i-1}) = \sup_{n \geq 0} \Pr\left( I = i \text{ and } J - n - 1 = 0 \,\middle\vert\, J - n - i \leq 0 \right) = \sup_{n \geq 0} \frac{a_{i, n+1}}{\sum_{k = 1}^\infty \sum_{\ell = 1}^{n + i} a_{k \ell}} =
			\sup_{j \geq 1} b_{i j}.
		$$
		Since $\Pr((I, J) \in \mathcal{G}) = 0$, the LAL gives us that for every $x \in \mathbb{R}_+$
		$$x\,<\, 1+ \sum_{i=1}^\infty P(\beta, \beta^{i-1}) x^i = 1 + \sum_{i = 1}^\infty \left(\sup_{j\geq 1} b_{ij}\right)x^i,$$
		as desired.
	\end{proof}
	
	\subsection{Choice functions}
	
	Our last example is a probabilistic application of the LAL. Let $U_1$, \ldots, $U_n$ be a collection of pairwise disjoint non-empty finite sets. A \emph{choice function} $F$ is a subset of $U_1 \cup \ldots \cup U_n$ such that $|F \cap U_k | = 1$ for all $1\leq k \leq n$. A \emph{partial choice function} $P$ is a subset of $U_1 \cup \ldots \cup U_n$ such that $|P \cap U_k | \leq 1$ for all $1\leq k \leq n$. For a partial choice function $P$ let
	$$\dom(P) \coloneqq \{1 \leq k \leq n\,:\,U_k \cap P \neq \emptyset\}.$$
	If $P$ is a partial choice function and $F$ is a choice function, then $P$ \emph{occurs} in $F$ if $P \subseteq F$.
	
	The language of choice functions is very natural for describing combinatrial problems. For example, if $G(V, E)$ is a graph, then we can assign to each vertex $v \in V$ a set $U_v \coloneqq \{(k, v) \,:\, 1 \leq k \leq r\}$. After that assign to each edge $uv \in V$ and $1 \leq k \leq r$ a partial choice function $P_{uv}^k \coloneqq \{(k, u), (k, v)\}$. Then a proper $r$-coloring of $G$ is nothing else but a choice function $F$ such that none of $P_{uv}^k$ occurs in $F$.
	
	A \emph{multichoice function} $M$ is any subset of $U_1 \cup \ldots \cup U_n$. Again, a partial choice function $P$ occurs in a multichoice function $M$ if $P \subseteq M$. Suppose that we are given a finite family $P_1$, \ldots, $P_m$ of non-empty ``forbidden'' partial choice functions. Clearly, there exists a choice function $F$ that avoids all of $P_1$, \ldots, $P_m$ if and only if there exists a multichoice function $M$ such that
	\begin{equation}\label{eq:multi}
		|M \cap U_k| \geq 1 + |\{1 \leq j \leq m\,:\, \dom(P_j)\ni k \text{ and } P_j \subseteq M\}|
	\end{equation}
	for all $1 \leq k \leq n$. Indeed, if $F$ is a choice function that avoids all of $P_1$, \ldots, $P_m$, then we can take $M = F$. If, on the other hand, $M$ satisfies (\ref{eq:multi}), then for each $1\leq k \leq n$ there is an element
	$$u_k \in (M \cap U_k) \setminus \left(\bigcup_{P_j \subseteq M} P_j\right).$$
	Then the set $\{u_1, \ldots, u_n\}$ is a choice function avoiding $P_1$, \ldots, $P_m$.
	
	The theorem that we are going to prove asserts that in fact it is enough to prove (\ref{eq:multi}) \emph{on average} for some random multichoice function $M$.
	
	\begin{theo}
		Let $U_1$, \ldots, $U_n$ be a collection of pairwise disjoint non-empty finite sets, and let $P_1$, \ldots, $P_m$ be a family of non-empty partial choice functions. Let $M_k \subseteq U_k$ be a random subset of $U_k$ for each $1 \leq k \leq n$. Suppose that the variables $M_k$ are mutually independent. Let $M \coloneqq \bigcup_{k=1}^n M_k$. If
		\begin{equation}\label{eq:multimean}
			\mathbb{E}|M \cap U_k| \geq 1 + \mathbb{E}|\{1 \leq j \leq m\,:\, \dom(P_j)\ni k \text{ and } P_j \subseteq M\}|			
		\end{equation}
		for all $1 \leq k \leq n$, then there exists a choice function $F$ in which none of $P_1$, \ldots, $P_m$ occur.
	\end{theo}
	\begin{proof}
		Let $[n]$ denote the set $\{1, \ldots, n\}$. For $u \in U_1 \cup \ldots \cup U_n$ let $p(u) \coloneqq \Pr(u \in M)$. Then
		$$\mathbb{E}|M \cap U_k| = \sum_{u \in U_k} p(u),$$
		and
		\begin{align*}
			\mathbb{E}|\{1 \leq j \leq m\,:\, \dom(P_j)\ni k \text{ and } P_j \subseteq M\}| = 
			\sum_{\dom(P_j)\ni k} \Pr(P_j \subseteq M) = \sum_{\dom(P_j)\ni k} \prod_{u \in P_j} p(u).
		\end{align*}
		So (\ref{eq:multimean}) says that
		\begin{equation}\label{eq:sums}
			\sum_{u \in U_k} p(u) \geq 1 + \sum_{\dom(P_j)\ni k} \prod_{u \in P_j} p(u).
		\end{equation}
		
		Let $\func(k) \coloneqq \sum_{u \in U_k} p(u)$ and $p'(u) \coloneqq p(u)/\func(k)$ for $1 \leq k \leq n$ and $u \in U_k$. Then (\ref{eq:sums}) can be rewritten as
		\begin{equation}\label{eq:sums1}
			\func(k) \geq 1 + \sum_{\dom(P_j)\ni k} \left(\prod_{u \in P_j} p'(u)\right) \left(\prod_{s \in \dom(P_j)}\func(s)\right).
		\end{equation}
		
		Let $\mathcal{X}$ be the set of all partial choice functions, and $\mathcal{G} \subseteq \mathcal{X}$ be the set of all partial choice functions in which none of $P_1$, \ldots, $P_m$ occur. Choose a random choice function $X$, picking an element $u \in U_k$ with probability $p'(u)$ and making choices for different $k$ independently. Then
		$$\Pr(P_j \subseteq X) = \prod_{u \in P_j} p'(u),$$
		so (\ref{eq:sums1}) turns into
		\begin{equation}\label{eq:sums2}
			\func(k) \geq 1 + \sum_{\dom(P_j)\ni k} \Pr(P_j \subseteq X) \left(\prod_{s \in \dom(P_j)}\func(s)\right).
		\end{equation}	
		
		We want to show that $\Pr(X \in \mathcal{G}) > 0$. Let $[n]$ denote the set $\{1, \ldots, n\}$. Then $\mathcal{P}([n])$ is a commutative monoid (with multiplication given by the union operation) generated by the set of singletons, that acts on $\mathcal{X}$ with
		$$S.P \coloneqq P \setminus \left(\bigcup_{k \in S} U_k\right).$$
		Clearly, $\mathcal{G}$ is closed under the $\mathcal{P}([n])$-action. Also note that $\Pr([n].X \in \mathcal{G}) = 1 > 0$.
		
		If $P$ is a partial choice function and $P \not \in \mathcal{G}$, but $k.P \in \mathcal{G}$, then for some $1 \leq j \leq m$ we have $P_j \subseteq P$ and $P_j \cap U_k \neq \emptyset$. Then let
		$$g_k(P) \coloneqq \dom(P_j).$$
		If there is more than one such $P_j$, then choose any one of them. Now for $1 \leq k \leq n$ and for $S \in \mathcal{P}([n])$ we have $P(k, S) \neq 0$ only if $k \in S$ and $S = \dom(P_j)$ for some $1\leq j \leq m$. In the latter case let $\Pi(S)$ be the set of all indices $1\leq j \leq m$ such that $S = \dom(P_j)$. Then
		\begin{align*}
		P(k, S) &= \sup_{V \subseteq [n]} \Pr\left( g_k\left(X\setminus \left(\bigcup_{k \in V} U_k\right)\right) = S \,\middle\vert\, X \setminus \left(\bigcup_{k \in V\cup S} U_k\right) \in \mathcal{G} \right)\\
		&\leq \sup_{V \subseteq [n]} \sum_{j \in \Pi(S)} \Pr\left( P_j \subseteq X\setminus \left(\bigcup_{k \in V} U_k\right) \,\middle\vert\, X\setminus\left(\bigcup_{k \in V\cup S} U_k\right) \in \mathcal{G} \right) \\
		& = \sum_{j \in \Pi(S)} \Pr(P_j \subseteq X).
		\end{align*}
		Hence (\ref{eq:LAL}) in this case turns into (\ref{eq:sums2}), and we are done.
	\end{proof}
	
	\section*{Aknowledgements}
	
	This work is supported by the Illinois Distinguished Fellowship.

\end{document}